\documentclass[11pt,letterpaper]{amsart}
\usepackage[T1]{fontenc}
\usepackage[utf8]{inputenc}

\usepackage{amstext}
\usepackage{amsmath, amssymb, amsthm}
\usepackage{amssymb}
\usepackage{mathrsfs}
\usepackage[a4paper, total={6in, 9in}, centering]{geometry}
\usepackage{enumerate}
\usepackage{xcolor}
\usepackage{comment}
\usepackage{hyperref}
\usepackage{graphicx}
\usepackage{mathtools}
\mathtoolsset{showonlyrefs}

\newtheorem{theorem}{Theorem}[section]
\newtheorem{lemma}[theorem]{Lemma}
\newtheorem{proposition}[theorem]{Proposition}
\newtheorem{corollary}[theorem]{Corollary}
\newtheorem{remark}[theorem]{Remark}

\newtheorem{definition}[theorem]{Definition}

\newtheorem{question}[theorem]{Question}

\DeclareMathOperator{\loc}{loc}

\DeclareMathOperator{\dist}{dist}
\DeclareMathOperator{\diam}{diam}

\DeclareMathOperator{\Ric}{Ric}

\DeclareMathOperator{\Id}{Id}

\DeclareMathOperator{\dvoll}{dvol}
\DeclareMathOperator{\spec}{spec}
\DeclareMathOperator{\im}{im}

\DeclareMathOperator{\Hess}{Hess}
\DeclareMathOperator{\grad}{grad}

\DeclareMathOperator{\rank}{rank}

\newcommand{\R}{\mathbb R}
\newcommand{\N}{\mathbb N}

\newcommand{\CB}{\mathcal{B}}
\newcommand{\CC}{\mathcal{C}}

\newcommand{\CG}{\mathcal{G}}

\newcommand {\RL}{\mathrm L}

\newcommand{\RW}{\mathrm W}
\newcommand{\RC}{\mathrm C}

\renewcommand{\le}{\leqslant}
\renewcommand{\ge}{\geqslant}

\renewcommand{\phi}{\varphi}

\renewcommand{\S}{\mathbb S}
\newcommand{\dvol}{\,\dvoll}

\newcommand{\bo}\boldsymbol{}

\newcommand{\class}{\mathscr{S}}

\title[Canonical foliation of bubblesheets]{Canonical foliation of bubblesheets}
\author{Jean Lagac\'{e}}
\address{Department of Mathematics, King's College London, Strand, London, WC2R 2LS, UK}
\email{jean.lagace@kcl.ac.uk}
\author{Stephen Lynch}
\email{stephen.lynch@kcl.ac.uk}

\begin{document}

\begin{abstract}
    We introduce a new curvature condition for high-codimension submanifolds of a Riemannian ambient space, called quasi-parallel mean curvature (QPMC). The class of submanifolds with QPMC includes all CMC hypersurfaces and submanifolds with parallel mean curvature. We use our notion of QPMC to prove that certain kinds of high-curvature regions which appear in geometric flows, called bubblesheets, can be placed in a suitable normal form. This follows from a more general result asserting that the manifold $\R^k \times \S^{n-k}$, equipped with any metric which is sufficiently close to the standard one, admits a canonical foliation by embedded $(n-k)$-spheres with QPMC.    
\end{abstract}

\maketitle

\section{Introduction}

At the core of differential geometry is the notion that the important features of a space should remain invariant under changes of coordinates. However a space with some special structure may admit preferred coordinate systems, which reveal its features with particular clarity. A proven method for finding such distinguished parameterizations is to identify a foliation of the space by submanifolds which are canonically determined by its geometry.

Foliations by hypersurfaces with constant mean curvature (CMC) have been used to parameterize the ends of asymptotically flat manifolds, leading to a definition of center of mass for isolated gravitating systems \cite{Huisken--Yau} (see also the more recent \cite{Cederbaum--Sakovich} and references therein). In the Lorentzian setting, a foliation by maximal (vanishing mean curvature) hypersurfaces played a role in the first proof of the stability of the Minkowski spacetime \cite{Christodoulou--Klainerman}. In yet another direction, a canonical CMC foliation for `necks' (Riemannian manifolds which locally resemble $\R \times \S^{n-1}$ at some scale), constructed by Hamilton, has been employed to continue geometric flows through neck singularities via surgery \cite{Hamilton_PIC, Huisk-Sin09}. 

In this paper we construct a new canonical foliation for another kind of geometry which forms near singularities of both the Ricci flow and the mean curvature flow. These geometries, called bubblesheets, locally resemble $\R^k \times \S^{n-k}$ at some scale, where $k\geq 2$. Bubblesheets are expected to be the most prevalent high-curvature regions in the mean curvature flow of hypersurfaces of dimension $n \geq 3$, and in the Ricci flow on manifolds of dimension $n \geq 4$. The essential question we address is the following: 

\begin{question}
Consider the manifold $\mathbb{R}^k \times \S^{n-k}$ equipped with the standard product metric $g_0 := g_{\R^k} \oplus g_{\S^{n-k}}$. This space is foliated by round, totally geodesic $(n-k)$-spheres, namely the slices $\{z\}\times\S^{n-k}$. If $g_0$ is deformed to a nearby metric, can these totally geodesic spheres be deformed into a canonical foliation of the resulting space?
\end{question}

The foliation we seek is by spheres of codimension $k \geq 2$ which, by analogy with Hamilton's CMC foliation for $k =1$, ought to solve some equation involving their mean curvature. The main challenge is to determine the right equation, or in this case \emph{system} of equations, since our spheres have codimension greater than one. One might seek a foliation by leaves whose mean curvature has constant length, or is parallel in the normal bundle. Neither of these conditions works; in a certain sense they are underdetermined and overdetermined respectively. The correct approach, we argue, is to foliate by spheres which have quasi-parallel mean curvature (QPMC). This new curvature condition is more general than having parallel mean curvature, but almost as restrictive (as it must be if our foliation is to be canonical).

To see what we mean by QPMC, fix a Riemannian manifold $(M,g)$ of dimension $n$ and consider a submanifold $\Sigma \subset M$ of codimension $k$. The Levi-Civita connection of $(M,g)$ induces a connection $\nabla^\perp$ and hence a Laplacian $\Delta^\perp$ acting on sections of the normal bundle $N\Sigma$. The operator $-\Delta^\perp$ admits a complete set of eigensections with nonnegative eigenvalues $\lambda_m = \lambda_m(\Sigma, g)$, labelled in nondecreasing order, repeating according to multiplicity. We write $P_\lambda$ for the $\RL^2$-orthogonal projection onto the eigenspace associated with $\lambda \in \spec(-\Delta^\perp)$ and define 
    \[Q := \sum_{\substack{
    \lambda \in \spec(-\Delta^\perp) \\ \lambda < \lambda_{k+1}
    }} P_{\lambda}.\]
A section of $N\Sigma$ is called \textbf{quasi-parallel} if it lies in $\im(Q)$, and so the submanifold $\Sigma$ is said to have QPMC if its mean curvature vector $H$ satisfies
    \begin{equation}\label{QPMC}
    (1 - Q)(H) = 0.
    \end{equation}

We note that $1-Q$ is a nonlocal pseudodifferential operator acting on sections of $N\Sigma$. Expressed as $H = Q(H)$, \eqref{QPMC} is a weakly elliptic quasilinear system for the position of the submanifold, with nonlocal right-hand side. From a variational point of view, $\Sigma$ has QPMC if and only if $\Sigma$ is a critical point for the volume functional with respect to variations whose initial velocity lies in $\ker(Q)$.

Let us further clarify the meaning of \eqref{QPMC} with a few remarks: 
\begin{itemize}
    \item If $\Sigma$ is of codimension 1 and 2-sided, then $Q$ is the projection onto the 1-dimensional space of harmonic sections of $N\Sigma$ (normal vectors of constant length). Therefore, in this setting $\Sigma$ has CMC if and only if it has QPMC.
    \item Since $\Sigma$ is compact, $\ker(\Delta^\perp)$ is precisely the space of parallel sections of $N\Sigma$. In particular, if $H$ is parallel in $N\Sigma$, meaning $\nabla^\perp H =0$, then $\Sigma$ has QPMC.
    \item When $k \geq 2$, for a general submanifold $\Sigma$, the normal bundle $N\Sigma$ may admit no global nonzero parallel sections. However, if $\lambda_k < \lambda_{k+1}$ (which is true generically), then $\im(Q)$ is $k$-dimensional. In a sense which can be made precise using Rayleigh quotients, $\im(Q)$ is the $k$-dimensional subspace of normal vector fields which are `as close as possible' to being parallel.
\end{itemize}
    
We now state our first main result. Recall the notation $g_0 = g_{\R^k} \oplus g_{\S^{n-k}}$.

\begin{theorem}\label{main entire}
Let $g$ be a smooth Riemannian metric on $M = \R^k \times \S^{n-k}$ and fix constants $\ell \geq 3$ and $\gamma \in (0,1)$. If $\delta>0$ is sufficiently small depending on $n$, $\ell$ and $\gamma$ then there exists a positive constant $\varepsilon = \varepsilon(n,\ell,\gamma,\delta)$ with the following property. If $\|g - g_0\|_{\RC^{\ell-1, \gamma}} \leq \varepsilon$ then $M$ admits a unique smooth foliation by embedded $(n-k)$-spheres, each of which has QPMC and is $\delta$-close to some $\{z\}\times\S^{n-k}$ in the graphical $\RC^{\ell,\gamma}$-sense.
\end{theorem}

\begin{remark}
In Section~\ref{section example} we provide examples of metrics on $\R^2 \times \S^1$ which admit a QPMC foliation by Theorem~\ref{main entire}, but do not admit a foliation whose leaves have parallel mean curvature. These examples are obtained from bumpy perturbations of almost collapsed Berger metrics on $\S^3$. 
\end{remark}

Using Theorem~\ref{main entire}, we produce a canonical QPMC foliation for more general geometries which resemble $\R^k \times \S^{n-k}$ only locally and after rescaling. Before stating this result we introduce some definitions. Let $(M,g)$ be a complete Riemannian $n$-manifold. The following notion of an $(\varepsilon, L, n-k)$-cylindrical region makes precise what we informally refer to as a bubblesheet. 

\begin{definition}\label{cylindrical region}
An open subset $\CC \subset M$ is $(\varepsilon, L, n-k)$-cylindrical if for each $p \in \CC$ there exists a scale $r(p) > 0$ and an embedding
    \[F_p : B^k(0, L) \times \S^{n-k} \to M\]
such that $F_p^*(r(p)^{-2}g)$ is $\varepsilon$-close to $g_0$ in $\RC^{[1/\varepsilon]}$.
\end{definition}

We will also need to refer to the following class of submanifolds of $M$.

\begin{definition}\label{delta vertical}
Let $\Sigma \subset M$ be a closed submanifold. We introduce a scale $\bar r$ defined so that $\pi \bar r$ is the average injectivity radius of $(M,g)$ over points in $\Sigma$. We say that $\Sigma$ is $\delta$-vertical if $\diam_g(\Sigma)/\bar r \leq 10\pi$ and 
    \[\bar r|A_\Sigma| + \bar r^2 |\nabla^\perp A_\Sigma| + \bar r^3 |\nabla^\perp \nabla^\perp A_\Sigma| \leq \delta,\]
where $A_\Sigma$ is the second fundamental form and $\nabla^\perp$ is the induced connection on $T^*\Sigma \otimes \dots \otimes T^*\Sigma \otimes N\Sigma$.
\end{definition}

Definitions \ref{cylindrical region} and \ref{delta vertical} are intended so that if $\Sigma$ is a $\delta$-vertical $(n-k)$-sphere embedded in an $(\varepsilon, L, n-k)$-cylindrical region and $\varepsilon$ is small, then $\delta$ controls how far $\Sigma$ is from a vertical sphere in $\R^k\times\S^{n-k}$ after pulling back. Note also that for $p \in \Sigma$ we have $\bar r/r(p) \to 1$ as $\varepsilon \to 0$. 

We now state our second main result. 

\begin{theorem}\label{main local}
Fix a constant $L \geq 1000$. Let $(M,g)$ be a complete Riemannian manifold and suppose $\CC \subset M$ is $(\varepsilon, L, n-k)$-cylindrical. If $\varepsilon$ is sufficiently small depending on $n$ and $\delta$ then there exists an open subset 
    \[\CC \subset \CB \subset M\]
which is foliated by $\delta$-vertical $(n-k)$-spheres with QPMC. Moreover, every embedded $(n-k)$-sphere in $M$ which is $\delta$-vertical, has QPMC and intersects $\CC$ is a leaf of this foliation. 
\end{theorem}

Theorem~\ref{main entire} implies the existence of a local QPMC foliation around any point in the region $\CC$ referred to in Theorem~\ref{main local}. The extra step required to prove the theorem is to show that these local foliations agree where they overlap, and hence combine to give a global foliation which covers all of $\CC$. This overlapping property is established via uniqueness results for QPMC leaves, which we prove in Section~\ref{section uniqueness}.

\subsection{Outline} After fixing notation in Section~\ref{notation section}, we compute in Section~\ref{section variations} the first variation of $(1-Q)(H)$ for a general submanifold of a Riemannian ambient space. Theorem~\ref{main entire} is proven using the implicit function theorem in Section~\ref{section existence}. We then prove Theorem~\ref{main local} in Section~\ref{section uniqueness}. In Section~\ref{section example} we demonstrate by example that certain metrics on $\R^2 \times \S^1$ do not admit a foliation by spheres with parallel mean curvature. Finally, in Section~\ref{section MCF}, we describe how bubblesheet regions occur close to singularities of the mean curvature flow, and show that Theorem~\ref{main local} can be used to cast these in a canonical normal form.


\section{Notation}\label{notation section}

Let $(M, g)$ be a Riemannian manifold. We write $\nabla$ for the Levi-Civita connection determined by $g$.  Our convention for the Riemann curvature tensor of $(M,g)$ is 
    \[R(U,V)W = \nabla_U \nabla_V W - \nabla_V\nabla_U W.\]

Let $\Sigma \subset M$ be a closed submanifold of codimension $k$. We introduce some standard notation and terminology concerning the geometry of $\Sigma$. First recall that for each $p \in \Sigma$, the tangent space $T_p M$ splits as 
    \[T_p M = T_p \Sigma \oplus (T_p \Sigma)^\perp = T_p \Sigma \oplus N_p \Sigma,\]
and this gives rise to the tangent bundle $T\Sigma$ and normal bundle $N\Sigma$. For a vector $V \in T_p M$, we write $V^\top$ and $V^\perp$ for the projections of $V$ onto $T_p\Sigma$ and $N_p\Sigma$ respectively. 

We use $h$ to denote the metric induced by $g$ on $\Sigma$. That is, 
    \[h(X, Y) = g(X, Y)\]
for $X$ and $Y$ tangent to $\Sigma$. The connection $\nabla$ induces connections $\nabla^\top$ and $\nabla^\perp$ on $T\Sigma$ and $N\Sigma$ respectively, via the formulae
\begin{equation}
    \nabla^\top_X Y := (\nabla_X Y)^\top \qquad \text{and} \qquad \nabla^\perp_X V := (\nabla_X V)^\perp.
\end{equation}
Note that $\nabla^\top$ is the Levi-Civita connection on $T\Sigma$ determined by $h$. The second fundamental form of $\Sigma$ acts on tangent vectors $X$ and $Y$  by
    \[A(X, Y) = (\nabla_{X} Y)^\perp.\]
If $X_i$ is a locally defined smooth frame for $T\Sigma$, then for $V \in \RC^\infty(N\Sigma)$ we have
    \[(\nabla_{X_i} V)^\top = -g(A_i^j, V)X_j,\]
and hence
    \[\nabla^\perp_i V = \nabla_{X_i} V + g(A^j_i, V) X_j,\]
where $A_i^j = h^{jk}A_{ki}$, $h^{jk}$ being the matrix inverse to $h_{jk}$. 

\subsection{The normal Laplacian} We define a Laplacian $\Delta^\perp : \RC^\infty(N\Sigma) \to \RC^\infty(N\Sigma)$ by
    \[\Delta^\perp  := -\nabla^{\perp*}\nabla^\perp,\]
where $\nabla^{\perp*} : \RC^\infty(T^*\Sigma \otimes N\Sigma) \to \RC^\infty(N\Sigma)$ is the $\RL^2$-adjoint of $\nabla^\perp$. With respect to any local frame $X_i$ for $T\Sigma$ we have 
    \[\Delta^\perp V = h^{ij}\left(\nabla^\perp_i(\nabla^\perp_j V) - \nabla^\perp_{\nabla^\top_i X_j} V\right).\]
The operator $\Delta^\perp$ is essentially self-adjoint on its domain, the Sobolev space $\RW^{2,2}(N\Sigma)$. Since $-\Delta^\perp$ is a positive elliptic operator on a compact manifold, it has compact resolvent, and so its spectrum is comprised of eigenvalues; these form a nondecreasing
sequence 
\begin{equation}
\label{eq:eigenvalues_increasing}
    0 \le \lambda_1(\Sigma,g) \le \lambda_2(\Sigma,g) \le \lambda_3(\Sigma,g) \le \dotso \nearrow \infty.
\end{equation}
To each $\lambda \in \spec(-\Delta^\perp)$ we have an associated $\RL^2$-projection $P_\lambda$ and eigenspace $\im(P_\lambda)$.
Moreover, we have the spectral decomposition
\begin{equation}
    -\Delta^\perp = \sum_{\lambda \in \spec(-\Delta^\perp)} \lambda P_\lambda.
\end{equation}
We note that
\begin{equation}
    \dim(\ker(\Delta^\perp)) \le k,
\end{equation}
with equality if and only if $N\Sigma$ admits a global frame consisting of parallel sections. 

For $m \in \N$, let $\mathbf{Gr}_m$ be the Grassmannian of 
$m$-dimensional subspaces of $\RC^\infty(N\Sigma)$. The eigenvalues of $-\Delta^\perp$ are characterized variationally by
\begin{equation}
    \lambda_m(\Sigma,g) = \min_{E \in \mathbf{Gr}_m} \max_{V \in E, \, \|V\|_{\RL^2} = 1} \int_\Sigma |\nabla^\perp V|^2 \dvoll_\Sigma,
\end{equation}
and any normal vector field $V$ which realises this minmax is an eigensection of $-\Delta^\perp$. It follows from this variational characterization that $\lambda_m(\Sigma,g)$ is piecewise smooth (and smooth away from points at which its multiplicity changes)
with respect to $\Sigma$ and $g$. 

\subsection{Quasi-parallel sections} 
For a general submanifold $\Sigma$, the normal bundle $N\Sigma$ may not admit any nonzero sections which are parallel with respect to $\nabla^\perp$. Therefore, in place of parallel sections, we work with a class of sections which behave almost as well and are far more abundant. These we call quasi-parallel. Recalling that $k$ is the codimension of $\Sigma$ in $M$, we define 
    \[Q := \sum_{\substack{\lambda \in \spec(-\Delta^\perp) \\ \lambda < \lambda_{k+1}}} P_\lambda,\]
and say that a section of $N\Sigma$ is quasi-parallel if it lies in $\im(Q)$.

Concerning this definition, a few remarks are in order. Note first that 
    \[\dim(\im(Q)) \leq k\]
with equality if and only if $\lambda_k < \lambda_{k+1}$. Moreover, for a generic submanifold $\Sigma$ and any $m$, the eigenvalue $\lambda_m$ is simple. Appealing to this fact in the case $m = k$ we see that for a generic $\Sigma$ we have $\lambda_k < \lambda_{k+1}$. Since this is an open condition, $Q$ has constant rank equal to $k$ on a dense open subset of the compact $(n-k)$-dimensional submanifolds of $M$. 

Next we observe that $Q$ varies smoothly while $\Sigma$ varies smoothly with $\lambda_k < \lambda_{k+1}$, using the variational characterisation of $\im(Q)$, which is as follows. Suppose $Q$ has rank $k$, i.e. $\lambda_k < \lambda_{k+1}$, and consider the function
    \[(V_1, \dots, V_k) \mapsto \sum_{j=1}^k \int_\Sigma |\nabla^\perp V_j|^2 \dvoll_\Sigma\]
acting on $\RL^2$-orthonormal collections of sections $V_1, \dots, V_k \in \RC^\infty(N\Sigma)$. This function is smooth and depends only on the subspace spanned by $V_1, \dots, V_k$, so it lifts to a smooth function defined on $\mathbf{Gr}_k$. Whenever $\lambda_k < \lambda_{k+1}$, the resulting function attains its unique global minimum at $\im(Q)$, and this minimum is nondegenerate. Therefore, $\im(Q)$ traces out a smooth path in $\mathbf{Gr}_k$ as $\Sigma$ varies smoothly. 

Finally, let us point out that $Q$ has rank $k$ if $N\Sigma$ admits a global frame consisting of parallel sections, or if $\Sigma$ is a perturbation of such a submanifold. This is, for example, the case when $M = \mathbb{R}^k\times\S^{n-k}$ with the metric $g_0$ and $\Sigma$ is a vertical sphere $\{z\}\times\S^{n-k}$, which is the setting with which we are principally concerned. Indeed, we then have 
    \[\lambda_1 = \dots =\lambda_k =0, \qquad \lambda_{k+1} = n-k,\]
so $\lambda_k < \lambda_{k+1}$ remains true for perturbations of $\Sigma$ and, we note, in this perturbed regime,
    \[Q = \sum_{\substack{\lambda \in \spec(-\Delta^\perp) \\ \lambda < \frac{n-k}{2}}} P_\lambda.\]

\subsection{Regularity of eigensections}

With minor modifications, all of the definitions in this section still make sense if $g$ is of class $\RC^{2,\gamma}$ and $\Sigma$ is of class $\RC^{3,\gamma}$. For later use, we record here a regularity result for the eigensections of $-\Delta^\perp$ under these hypotheses. 

\begin{lemma}\label{eigensection regularity}
Suppose $g$ is of class $\RC^{\ell-1,\gamma}$ and $\Sigma$ is of class $\RC^{\ell,\gamma}$, where $\ell \geq 3$ and $\gamma \in (0,1)$. The eigensections of $-\Delta^\perp$ are then of class $\RC^{\ell-1,\gamma}$. 
\end{lemma}
\begin{proof}
Let $U$ be an eigensection of $-\Delta^\perp$ with eigenvalue $\lambda$. We consider a family of smooth submanifolds $\Sigma_s$ which converge to $\Sigma$ in $\RC^{\ell,\gamma}$ as $s \to 0$. Let us write $\Delta_s^\perp$ for the normal Laplacian on $N\Sigma_s$ and define $E_s$ to be the span of all the eigensections of $-\Delta_s^\perp$ with eigenvalue in $(\lambda-\delta, \lambda+\delta)$. If $\delta$ is small enough then the only element of $\spec(-\Delta^\perp)$ in this interval is $\lambda$. It then follows that $\dim(E_s)$ is constant and equal to $m := \rank(P_\lambda)$ for small $s$, and $E_s$ converges to $\im(P_\lambda)$ in the Grassmannian $\mathbf{Gr}_m$ of $\RL^2(N\Sigma)$ as $s \to 0$. 

We now demonstrate that every element of $\im(P_\lambda)$ is of class $\RC^{\ell-1,\gamma}$, by arguing that the $\RC^{\ell-1,\gamma}$-norm of any normalised eigensection in $E_s$ is bounded independently of $s$. To that end, let $\tilde U$ be an eigensection of $-\Delta^\perp_s$ with eigenvalue $\tilde \lambda \in (\lambda-\delta, \lambda+\delta)$. Locally, we may choose a frame $N_a$ for $N\Sigma_s$ and coordinates $x^i$ for $\Sigma_s$, of class $\RC^{\ell-1,\gamma}$ and $\RC^{\ell, \gamma}$, respectively. Suppressing dependencies on $s$, we then have
    \[\Delta^\perp \tilde U = (\Delta \tilde U^a)N_a + 2 h^{ij} \frac{\partial\tilde U^a}{\partial x^i}\nabla_j^\perp N_a + \tilde U^a \Delta^\perp N_a.\]
Let us define $q_{ab} = g(N_a, N_b)$ and write $q^{ab}$ for the matrix inverse to $q_{ab}$. Relabelling indices, we see that the eigenvalue equation for $\tilde U$ is equivalent to the following system of equations for the components $\tilde U^a$:
    \[\Delta \tilde U^a + 2 q^{ac} h^{ij} g(\nabla_j^\perp N_b, N_c) \frac{\partial \tilde U^b}{\partial x^i} + q^{ac} g(\Delta^\perp N_b, N_c) \tilde U^b + \tilde \lambda \tilde U^a= 0.\]
Inspection shows that 
    \[\|N_a\|_{\RC^{\ell-1,\gamma}}, \qquad \|g(\nabla_j^\perp N_b, N_c)\|_{\RC^{\ell-2,\gamma}} \qquad \text{and} \qquad \|g(\Delta^\perp N_b, N_c)\|_{\RC^{\ell-3,\gamma}}\]
are bounded independently of $s$. We conclude that the components $\tilde U^a$ solve a system of the form 
    \[h^{ij} \frac{\partial^2 \tilde U^a}{\partial x^i \partial x^j} + B^{ai}_b\frac{\partial \tilde U^b}{\partial x^i} + C^a_b \tilde U^b = 0\]
where the norms 
    \[\|h^{ij}\|_{\RC^{\ell-1,\gamma}}, \qquad \|B^{ai}_b\|_{C^{\ell-2,\gamma}} \qquad  \text{and} \qquad \|C^a_b\|_{\RC^{\ell-3,\gamma}}\]
are bounded independently of $s$. Using the Schauder estimates (see e.g. \cite[Section~5.5]{Giaquinta}) and an appropriate covering of $\Sigma_s$ by small balls we conclude that $\|\tilde U\|_{\RC^{\ell-1,\gamma}}$ is bounded independently of $s$, provided we have a uniform bound for $\sup_\Sigma |\tilde U|$. Such a bound follows from a standard iteration argument using $\Delta |\tilde U| \geq -\tilde \lambda|\tilde U|$ and the Sobolev inequality, provided that we normalise $\|\tilde U\|_{\RL^2} = 1$. This completes the proof. 
\end{proof}

\section{Variation Formulae}\label{section variations}

Let $(M,g)$ be a complete Riemannian $n$-manifold and $\Sigma$ a compact submanifold of codimension $k$. Throughout this section we consider a smooth variation of $\Sigma$, given by a smooth family of submanifolds $\Sigma_s$ such that $\Sigma_0 = \Sigma$. We parameterize the variation by a family of embeddings
    \[
    F : \Sigma \times (-\varepsilon, \varepsilon) \to M
    \]
such that $F(\cdot, 0)$ is the inclusion $\Sigma \hookrightarrow M$. We also assume that $\tfrac{\partial F}{\partial s}(\cdot, 0)$ is normal to $\Sigma$; while this assumption is not strictly necessary it does simplify some formulae and can always be arranged by composing $F$ with an $s$-dependent family of diffeomorphisms of $\Sigma$. Let us write 
    \[V := \frac{\partial F}{\partial s}(\cdot, 0).\]

Let $W = W(x,s)$ be a smoothly varying vector field on $\Sigma_s$. We define $\nabla_s W$ to be the covariant derivative of $W$ along the curve $s \mapsto F(x,s)$. In case $W(x,s) = \overline W(F(x,s))$ for some vector field $\overline W$ on $M$, we have 
    \[\nabla_s W = \nabla_{\frac{\partial F}{\partial s}} \overline W.\]
When $W(x,s)$ is normal to $\Sigma_s$ for each $s$, we define
    \[\nabla_s^\perp W = (\nabla_s W)^\perp.\]
    
We recall the first variation formula for the mean curvature.

\begin{lemma}\label{1st variation MC}
At $s = 0$ we have
    \[\nabla^\perp_s H  = \Delta^\perp V+ g(A^i_{j}, V)A^j_i + h^{ij} (R(V, X_i)X_j)^\perp\]
with respect to any local frame $X_i$ for $T\Sigma$.
\end{lemma}

Our goal is to compute the first variation of $(1-Q)(H)$, in the situation where $g$ and $\Sigma$ are such that $\lambda_k(\Sigma, g) < \lambda_{k+1}(\Sigma,g)$. This assumption ensures that for $s$ sufficiently small the operator $Q$ associated with $\Sigma_s$ varies smoothly in $s$. 

Before proceeding we introduce one more instance of notation. Suppose $W$ is a normal vector field on $\Sigma$. We then define 
    \begin{equation}\label{def_B}
        \Lambda(V,W) := (\nabla_s^\perp \Delta^\perp  - \Delta^\perp \nabla_s^\perp )\tilde W\big|_{s = 0} 
    \end{equation}
for an arbitrary extension $\tilde W(\cdot, s)$ which is smooth, normal to $\Sigma_s$, and agrees with $W$ when $s = 0$. The extension we choose is irrelevant. Indeed, we show in Appendix~\ref{appendix commutators} that $\Lambda(V,W)$ is given by an expression in $W$ and its first- and second-order spatial derivatives, with coefficients determined by $V$, $\Sigma$ and $g$. 

\begin{proposition}\label{var P}
Suppose $g$ and $\Sigma$ are such that $\lambda_k(\Sigma, g) < \lambda_{k+1}(\Sigma,g)$. Let $W = W(\cdot,s)$ be a smooth vector field normal to $\Sigma_s$. When $s = 0$ we have
    \begin{align*}
        \nabla_s^\perp Q(W) &= Q(\nabla_s^\perp W) + \sum_{p > k} \sum_{m = 1}^k \frac{1}{\lambda_p - \lambda_m} \langle (1-Q)(W), U_p\rangle_{\RL^2} \langle \Lambda(V,U_m), U_p\rangle_{\RL^2} U_m\\
        &\qquad+ \sum_{p > k} \sum_{m = 1}^k \frac{1}{\lambda_p - \lambda_m} \langle Q(W),  U_m \rangle_{\RL^2}\langle \Lambda(V,U_m), U_p\rangle_{\RL^2}U_p\\
        &\qquad - \sum_{m = 1}^k \bigg(\int_{\Sigma} g((1-Q)(W), U_m) g(H, V) \dvoll_{\Sigma}\bigg) U_m
    \end{align*}
where $U_m$ is an $\RL^2$-orthonormal basis of eigensections for $-\Delta^\perp$ at $s = 0$, labelled so that the eigenvalues $\lambda_m$ are nondecreasing in $m$. 
\end{proposition}
\begin{proof}
Let $U_m = U_m(\cdot,s) \in \RC^\infty(N\Sigma_s)$ be an orthonormal basis for for $\RL^2(N\Sigma_s)$ which varies smoothly in $s$. We assume that, for each $s$, $\{U_{m}\}_{m = 1}^k$ spans $\im(Q)$ and $\{U_m\}_{m = k+1}^\infty$ spans $\ker(Q)$. In addition, we assume that when $s = 0$ each $U_m$ is an eigensection for $-\Delta^\perp$, labelled so that the corresponding eigenvalues $\lambda_m$ are nondecreasing in $m$. (Note that we cannot in general demand that the $U_m$ are eigensections for every $s$, since they might then fail to vary smoothly in $s$.) 

We have
    \begin{equation}\label{eq:PinaBasis}
        Q = \sum_{m = 1}^k \langle \cdot, U_{m}\rangle_{\RL^2} U_{m},
    \end{equation}
and hence
    \begin{align*}
        \nabla_s^\perp Q(W) &= Q(\nabla_s^\perp W) + \sum_{m = 1}^k\langle W, \nabla_s^\perp U_m \rangle_{\RL^2} U_m + \sum_{m = 1}^k\langle W,  U_m \rangle_{\RL^2} \nabla_s^\perp U_m\\
        &\qquad- \sum_{m = 1}^k \bigg(\int_{\Sigma_s} g(W, U_m) g(H, \tfrac{\partial F}{\partial s}) \dvoll_{\Sigma_s}\bigg)U_m.
    \end{align*}
Differentiating the orthogonality condition
    \begin{equation}
        \langle U_m, U_p \rangle_{\RL^2} = \delta_{mp}
    \end{equation}
with respect to $s$ yields
    \begin{equation}\label{var P orthogonality}
        0 = \langle \nabla_s^\perp U_m, U_p\rangle_{\RL^2} + \langle U_m, \nabla_s^\perp U_p\rangle_{\RL^2} - \int_{\Sigma_s} g(U_m, U_p) g(H, \tfrac{\partial F}{\partial s}) \dvoll_{\Sigma_s}.
    \end{equation}
Multiplying this identity by $\langle W, U_p \rangle_{\RL^2} U_m$, summing over both indices (from 1 to $k$) and using \eqref{eq:PinaBasis}, we see that that
\begin{equation}
    \begin{aligned}
        0 &= \sum_{m = 1}^k\langle W, Q(\nabla_s^\perp U_m) \rangle_{\RL^2} U_m + \sum_{m = 1}^k\langle W,  U_m \rangle_{\RL^2} Q(\nabla_s^\perp U_m)\\
        & \qquad - \sum_{m = 1}^k \bigg(\int_{\Sigma_s} g(U_m, Q(W)) g(H, \tfrac{\partial F}{\partial s}) \dvoll_{\Sigma_s}\bigg) U_m
    \end{aligned}
\end{equation}
and hence 
    \begin{align}\label{var Q first step}
        \nabla_s^\perp Q(W) &= Q(\nabla_s^\perp W) + \sum_{m = 1}^k\langle W, (1-Q)(\nabla_s^\perp U_m) \rangle_{\RL^2} U_m + \sum_{m = 1}^k\langle W,  U_m \rangle_{\RL^2} (1-Q)(\nabla_s^\perp U_m) \notag\\
        &\qquad - \sum_{m = 1}^k \bigg(\int_{\Sigma_s} g((1-Q)(W), U_m) g(H, \tfrac{\partial F}{\partial s}) \dvoll_{\Sigma_s}\bigg) U_m.
    \end{align}

We now rewrite the terms $(1-Q)(\nabla_s^\perp U_m)$ for $m \leq k$. Consider any $p > k$, so that $U_p$ lies in $\ker(Q)$. We first observe that 
    \begin{align*}
        \frac{d}{ds}\langle\Delta^\perp U_m, U_p\rangle_{\RL^2} &= \langle\nabla_s^\perp \Delta^\perp U_m, U_p\rangle_{\RL^2} + \langle \Delta^\perp U_m, \nabla_s^\perp U_p\rangle_{\RL^2}\\
        &\qquad- \int_{\Sigma_s} g(\Delta^\perp U_m, U_p) g(H, \tfrac{\partial F}{\partial s}) \dvoll_{\Sigma_s}.
    \end{align*}
Since $U_m \in \im(Q)$ we have $\Delta^\perp U_m \in \im(Q)$, whereas $U_p \in \ker(Q)$, so the left-hand side is zero and we are left with
    \[\langle\nabla_s^\perp \Delta^\perp U_m, U_p\rangle_{\RL^2} = - \langle \Delta^\perp U_m, \nabla_s^\perp U_p\rangle_{\RL^2}+\int_{\Sigma_s} g(\Delta^\perp U_m, U_p) g(H, \tfrac{\partial F}{\partial s}) \dvoll_{\Sigma_s},\]
which may also be expressed as 
    \begin{align*}
        \langle(\nabla_s^\perp \Delta^\perp - \Delta^\perp \nabla_s^\perp) U_m, U_p\rangle_{\RL^2} &= - \langle \nabla_s^\perp U_m, \Delta^\perp U_p\rangle_{\RL^2} - \langle \Delta^\perp U_m, \nabla_s^\perp U_p\rangle_{\RL^2}\\
        &\qquad +\int_{\Sigma_s} g(\Delta^\perp U_m, U_p)g(H, V)\dvoll_{\Sigma_s}.
    \end{align*}
Evaluating this identity at $s = 0$ gives
    \begin{align*}
        \langle \Lambda(V,U_m), U_p\rangle_{\RL^2} &= \lambda_p\langle \nabla_s^\perp U_m, U_p\rangle_{\RL^2} + \lambda_m \langle U_m, \nabla_s^\perp U_p\rangle_{\RL^2} \\
        &\qquad-\lambda_m \int_{\Sigma} g(U_m, U_p)g(H, V)\dvoll_{\Sigma},
    \end{align*}
and after inserting \eqref{var P orthogonality} we arrive at
    \begin{align*}
        \langle \Lambda(V,U_m), U_p\rangle_{\RL^2} &=  (\lambda_p - \lambda_m)\langle \nabla_s^\perp U_m, U_p\rangle_{\RL^2} 
    \end{align*}
Since this holds for every $p > k$ we conclude that, when $s = 0$,
    \[(1-Q)(\nabla^\perp_s U_m) = \sum_{p > k} \frac{1}{\lambda_p - \lambda_m} \langle \Lambda(V,U_m), U_p\rangle_{\RL^2}U_p.\]

Inserting the previous identity into \eqref{var Q first step} shows that when $s =0$ we have
    \begin{align*}
        \nabla_s^\perp Q(W) &= Q(\nabla_s^\perp W) + \sum_{p > k} \sum_{m = 1}^k \frac{1}{\lambda_p - \lambda_m} \langle W, U_p\rangle_{\RL^2} \langle \Lambda(V,U_m), U_p\rangle_{\RL^2} U_m\\
        &\qquad+ \sum_{p > k} \sum_{m = 1}^k \frac{1}{\lambda_p - \lambda_m} \langle W,  U_m \rangle_{\RL^2}\langle \Lambda(V,U_m), U_p\rangle_{\RL^2}U_p\\
        &\qquad - \sum_{m = 1}^k \bigg(\int_{\Sigma} g((1-Q)(W), U_m) g(H, V) \dvoll_{\Sigma}\bigg) U_m.
    \end{align*}
In the first two sums on the right, $W$ may be replaced with $(1-Q)(W)$ and $Q(W)$ respectively. This completes the proof.
\end{proof}

We have the following consequence of Lemma~\ref{1st variation MC} and Proposition~\ref{var P}. 

\begin{corollary}
Suppose $g$ and $\Sigma$ are such that $\lambda_k(\Sigma, g) < \lambda_{k+1}(\Sigma,g)$. We assume in addition that $\Sigma$ has QPMC, i.e., $(1-Q)(H) = 0$. When $s = 0$ we then have
    \begin{align}\label{QPMC variation}
        \nabla_s^\perp (1-Q)(H) &= (1-Q)(\Delta^\perp V + g(A^i_{k}, V)A^k_i + h^{ij} (R(V, X_i)X_j)^\perp)\notag \\
        &\qquad-\sum_{p > k} \sum_{m = 1}^k \frac{1}{\lambda_p - \lambda_m} \langle Q(H),  U_m \rangle_{\RL^2}\langle \Lambda(V,U_m), U_p\rangle_{\RL^2}U_p,
    \end{align}
where $U_m$ is an $\RL^2$-orthonormal basis of eigensections for $-\Delta^\perp$ at $s = 0$, labelled so that the eigenvalues $\lambda_m$ are nondecreasing in $m$. 
\end{corollary}

If $M = \R^k \times \S^{n-k}$, $g = g_0$ and $\Sigma = \{z\} \times \S^{n-k}$, the formula \eqref{QPMC variation} simplifies considerably. Indeed, $\Sigma$ is then totally geodesic and in particular minimal, so when $s = 0$ we have
    \[\nabla_s^\perp (1-Q)(H) = (1-Q)(\nabla_s^\perp H).\]
Moreover, the first variation of $H$ reduces to
    \[\nabla_s^\perp H = \Delta^\perp V,\]
so that, when $s = 0$,
    \[\nabla_s^\perp (1-Q)(H) = (1-Q)(\Delta^\perp V).\]
Given that $1-Q$ is simply the projection onto $\im(\Delta^\perp)$ in this particular case, we have:

\begin{corollary}\label{var (1-P)(H)}
If $M = \R^k \times \S^{n-k}$, $g = g_0$ and $\Sigma = \{z\}\times\S^{n-k}$ for some $z \in \mathbb{R}^k$, then 
    \[\nabla_s^\perp (1-Q)(H)  = \Delta^\perp V\]
when $s=0$.
\end{corollary}


\section{Foliation of $\mathbb{R}^k \times \S^{n-k}$ by QPMC spheres}\label{section existence}

Let $M$ denote the cylinder $\mathbb{R}^k \times \S^{n-k}$, where $n \geq 2$ and $1 \leq k \leq n-1$. Points in $M$ will be written as $(z, \omega)$. Let $g_{\S^{n-k}}$ be the round metric on $\S^{n-k}$ with sectional curvatures equal to one. Recall the notation 
    \[g_0 := g_{\R^k} \oplus g_{\S^{n-k}}.\]

We consider an integer $\ell \geq 3$ and real number $\gamma \in (0,1)$. We write $\CG$ for the space of Riemannian metrics on $M$ of class $\RC^{\ell-1, \gamma}$. Note that $\CG$ is an open cone in the Banach space $\RC^{\ell-1, \gamma}(T^*M \otimes T^*M)$.

Each $u \in \RC^{\ell,\gamma}(\S^{n-k};\R^k)$ gives rise to a graphical codimension-$k$ submanifold of $M$, namely 
    \[\Sigma_u := \{(z, \omega) : z = u(\omega)\}.\]
We consider $\Sigma_u$ to be identified with $\S^{n-k}$ via the embedding $(u(\omega), \omega)$. Given such a graphical submanifold $\Sigma_u$, and a metric $g$ on $M$, we write $N_a^{g, u}$ for the orthogonal projection of $\partial_{z^a}$ onto $N\Sigma_u$. For any choice of $g$ and $u$, the sections $\{N_a^{g, u}\}_{a=1}^k$ are a frame for $N\Sigma_u$. We denote by $\dvoll_{g,u}$ the induced volume form on $\Sigma_u$. For each ambient metric $g$, we write $H_{g, u}$ for the mean curvature vector of $\Sigma_u$ with respect to $g$. The projection onto the space of quasi-parallel sections of $N\Sigma_u$ is given by
    \[Q_{g,u} = \sum_{\substack{\lambda \in \spec(-\Delta^\perp_{g, u}) \\ \lambda < \lambda_{k+1}}} P_{g, u, \lambda},\]
where the operators $P_{g, u, \lambda}$ are the projections onto the eigenspaces of $\Delta^\perp_{g, u}$. 

We aim to find, for each metric $g \in \mathcal G$ which is sufficiently close to $g_0$, a map $u:\S^{n-k}\to\R^k$ such that $H_{g, u}$ is quasi-parallel, i.e.,
    \[(1-Q_{g, u}) (H_{g, u}) = 0.\]
This will be achieved using the implicit function theorem, applied to a suitable mapping, which we now construct. 
    
We begin by introducing a frame for $N\Sigma_u$ which consists of elements of $\im(Q_{g, u})$, namely
    \[E_a^{g, u} := Q_{g, u}(N_a^{g, u}).\]

\begin{lemma}\label{E is a frame}
If $\|g - g_0\|_{\RC^{\ell-1,\gamma}}$ and $\|u\|_{\RC^{\ell,\gamma}}$ are both sufficiently small then $\{E_a^{g,u}\}_{a=1}^k$ is a global frame for $N\Sigma_u$.
\end{lemma}
\begin{proof}
The projection $Q_{g, u}$ varies continuously as $g$ and $u$ vary continuously in a neighbourhood of $(g_0, 0)$. Moreover, the operator $Q_{g_0, 0}$ acts by the identity on the subspace spanned by $\{N^{g_0, 0}_a\}_{a =1}^k$. The claim follows. 
\end{proof}

We henceforth restrict attention to $g$ and $u$ such that Lemma~\ref{E is a frame} holds.

\begin{definition}\label{definition J}
We define
    \[J : \CG \times \RC^{\ell,\gamma}(\S^{n-k};\R^k) \to \RC^{\ell-2,\gamma}(\S^{n-k};\R^k)\]
such that $J(g, u)$ has components 
    \[J^a(g,u) = g((1-Q_{g, u}) (H_{g, u}), E^{g, u}_a) f_{g,u},\]
where $f_{g,u} : \S^{n-k} \to \mathbb{R}$ is the function determined by
    \[\dvoll_{g,u} = f_{g,u} \dvoll_{\S^{n-k}}.\]
\end{definition}

In defining $J$ as a map into $\RC^{\ell-2,\gamma}$ we have implicitly used the following facts: firstly, the mean curvature of a $\RC^{\ell,\gamma}$-submanifold is in $\RC^{\ell-2,\gamma}$ and, secondly, for $u \in \RC^{\ell, \gamma}$ the projection $Q$ maps $\RC^{\ell-2,\gamma}(N\Sigma_u)$ into itself by Lemma~\ref{eigensection regularity}.

In light of Lemma~\ref{E is a frame}, we have that $J(g, u) = 0$ if and only if $H_{g, u}$ is quasi-parallel. Note also that $J(g_0, 0) = 0$. We will solve the equation $J(g,u) = 0$ in a neighbourhood of $(g_0, 0)$ using the implicit function theorem. As it stands, the linearisation
    \[DJ_{(g_0,0)}(0, \cdot) : \RC^{\ell,\gamma}(\S^{n-k};\R^k) \to \RC^{\ell-2,\gamma}(\S^{n-k};\R^k)\]
fails to be invertible. Indeed, one can show (see Lemma~\ref{linearisation of J} below) that 
    \[DJ_{(g_0,0)}(0, \cdot) = \Delta_{\S^{n-k}},\]
which has $k$-dimensional kernel and cokernel consisting of the constant maps from $\S^{n-k}$ into $\mathbb{R}^k$. Geometrically, this failure of invertibility arises because the isometry group of $g_0$ contains the $k$-dimensional space of translations acting on $\R^k$.

The first step towards resolving this issue is to observe that $J$ actually takes values in the codimension-$k$ subspace of $\RC^{\ell-2,\gamma}(\S^{n-k};\R^k)$ consisting of maps with mean-zero components. This of course is by design and explains the introduction of the frame $E^{g,u}_a$ and weight $f_{g,u}$ in Definition~\ref{definition J}. Let us define, for each $\ell \in \mathbb{N}$, 
    \[\mathring{\RC}^{\ell,\gamma}(\S^{n-k};\R^k) := \bigg\{u \in  \RC^{\ell,\gamma} (\S^{n-k};\R^k): \int_{\S^{n-k}} u^a \, \dvoll_{\S^{n-k}} = 0 \; \forall \; a \in \{1, \dotsc, k\}\bigg\}.\] 

\begin{lemma}
The map $J$ sends $\CG \times \RC^{\ell,\gamma}(\S^{n-k};\R^k)$ to $\mathring{\RC}^{\ell-2,\gamma}(\S^{n-k};\R^k)$. 
\end{lemma}

\begin{proof}
Fix $g \in \CG$ and $u \in \RC^{\ell,\gamma}(\S^{n-k};\R^k)$. The claim is that for each $a \in \{1,\dotsc,k\}$ we have
    \[\int_{\S^{n-k}} J^a(g,u) \, \dvoll_{\S^{n-k}} = 0.\]
Unpacking the definition of $J^a(g,u)$ gives
    \begin{align*}
    \int_{\S^{n-k}} J^a(g,u) \, \dvoll_{\S^{n-k}} &= \int_{\S^{n-k}} g((1-Q)(H), E_a) f_{g,u} \dvoll_{\S^{n-k}}\\
    &= \int_{\S^{n-k}} g((1-Q)(H), E_a) \dvoll_{g,u}\\
    &= \langle (1-Q)(H), E_a\rangle_{\RL^2(N\Sigma_u)}.
    \end{align*}
Recall that $E_a$ lies in $\im(Q)$. The claim follows, since the images of $1-Q$ and $Q$ are orthogonal subspaces of $\RL^2(N\Sigma_u)$.
\end{proof}

To account for the nontrivial kernel of $DJ_{(g_0,0)}(0, \cdot)$ we only consider submanifolds $\Sigma_u$ which are appropriately centered about $\{0\}\times\S^{n-k}$. That is, we henceforth consider $J$ to be defined as before, but with its domain restricted, so that
    \[J: \CG \times \mathring{\RC}^{\ell,\gamma} (\S^{n-k};\R^k) \to \mathring{\RC}^{\ell-2,\gamma} (\S^{n-k};\R^k).\]

We now compute the linearisation of $J$ at $(g_0, 0)$ using the variation formulae derived in Section~\ref{section variations}.

\begin{lemma}\label{linearisation of J}
The linearisation 
    \[DJ_{(g_0, 0)}(0, \cdot) : \mathring{\RC}^{\ell,\gamma} (\S^{n-k};\R^k) \to \mathring{\RC}^{\ell-2,\gamma} (\S^{n-k};\R^k)\]
is given by 
    \[DJ_{(g_0, 0)}(0, \varphi) = \Delta_{\S^{n-k}} \varphi.\]
\end{lemma}

\begin{proof}
It suffices to prove the claim when $\varphi$ is smooth, since the general case then follows by approximation. Consider the smooth variation of $\{0\} \times \S^{n-k}$ given by 
    \[F(\omega, s) = (s\varphi(\omega), \omega), \qquad s \in (-\varepsilon, \varepsilon).\]
With respect to $g_0$ we have 
    \[\frac{\partial F}{\partial s}(\cdot, 0) = \varphi^a \partial_{z^a}.\]
We need to evaluate 
    \[\frac{\partial}{\partial s} J^a(g_0, s \varphi) = \frac{\partial}{\partial s} \Big(g_0((1-Q)(H), E_a) f\Big) \]
at $s = 0$. Here and in the computation that follows we often suppress dependencies on the ambient metric and $s$, simply writing $H$ for $H_{g_0, s\varphi}$, and so on. 

We first expand and use the first variation formula for the volume form of a submanifold to obtain
    \begin{align*}
        \frac{\partial}{\partial s} J^a(g_0, s \varphi) &= g_0( \nabla_s^\perp (1-Q)(H), E_a) f + g_0((1-Q)(H), \nabla_s^\perp E_a) f\\
        &\qquad - g_0((1-Q)(H), E_a) g_0(H, \tfrac{\partial F}{\partial s}) f.
    \end{align*}
When $s = 0$ we have $E_a = \partial_{z^a}$, $H = 0$ and $f = 1$, so this reduces to
    \[\frac{\partial}{\partial s} J^a(g_0, s \varphi) = g_0( \nabla_s^\perp (1-Q)(H), \partial_{z^a}).\]
By Corollary~\ref{var (1-P)(H)} we know that
    \[\nabla^\perp_s (1-Q)(H) = \Delta^\perp \frac{\partial F}{\partial s}=\Delta \varphi^a \partial_{z^a}\]
when $s = 0$. The claim follows.
\end{proof}

Lemma~\ref{linearisation of J} and standard results in linear elliptic theory give the following.

\begin{lemma}\label{linearisation invertible}
The linearisation 
    \[DJ_{(g_0, 0)}(0, \cdot) : \mathring{\RC}^{\ell,\gamma} (\S^{n-k};\R^k) \to \mathring{\RC}^{\ell-2,\gamma} (\S^{n-k};\R^k)\]
is invertible.
\end{lemma}

With all of these preparations concluded, we may now proceed with the main result of this section. 

\begin{proposition}\label{foliation existence 1}
For every $\ell \ge 3$ and $\gamma \in (0,1)$ there are constants $\delta_0 > 0$ and $\varepsilon > 0$ which depend only on $n$, $\ell$ and $\gamma$ and have the following property. If $\|g - g_0\|_{\RC^{\ell-1, \gamma}} \leq \varepsilon$ then for every $z \in \mathbb{R}^k$ there exists a unique  
    \[u_\star = u_\star(z, g) \in \{u \in \mathring{\RC}^{\ell,\gamma} (\S^{n-k};\R^k) : \|u\|_{\RC^{\ell,\gamma}} \leq \delta_0\}\]
such that the graphical submanifold $\Sigma_{z + u_\star}$ has QPMC with respect to $g$. Moreover, there is a constant $C = C(n, \ell, \gamma)$ such that
    \[\|u_\star(z, g)\|_{\RC^{\ell,\gamma}} \leq C \|g-g_0\|_{\RC^{\ell-1,\gamma}}.\]
\end{proposition}
\begin{proof}
We first perturb $\{0\} \times \S^{n-k}$ to a QPMC sphere. This is easily achieved using Lemma~\ref{linearisation invertible} and the implicit function theorem. Indeed, by those two results, there are positive constants $\delta_0$ and $\varepsilon$, each of which depends only on $n$, $\ell$ and $\gamma$, with the following property. If $\|g - g_0\|_{\RC^{\ell-1, \gamma}} \leq \varepsilon$ then for every $z \in \mathbb{R}^k$ there exists a unique  
    \[u_\star = u_\star(0, g) \in \{u \in \mathring{\RC}^{\ell,\gamma} (\S^{n-k};\R^k) : \|u\|_{\RC^{\ell,\gamma}} \leq \delta_0\}\]
which solves
    \[J(g, u_\star(0,g)) =0.\]
The implicit function theorem also provides the estimate
    \begin{equation}\label{bound u by g}
        \|u_\star(0, g)\|_{\RC^{\ell,\gamma}} \leq C(n,\ell,\gamma) \|g-g_0\|_{\RC^{\ell-1,\gamma}}.
    \end{equation}
Recall that $J(g, u_\star(0,g)) =0$ implies $\Sigma_{u_\star(0,g)}$ has QPMC with respect to $g$. 

Now consider an arbitrary $z_0 \in \mathbb{R}^k$. Let $\tau_{z_0} : M \to M$ denote the translation $(z,\omega) \mapsto (z+z_0,\omega)$. We then have
    \[J(\tau^*_{z_0} g, u) = J(g, z_0 + u)\]
for every $u$. For each $g \in \mathcal G$ such that $\|g - g_0\|_{\RC^{\ell-1, \gamma}} \leq \varepsilon$, we set 
    \[u_\star(z_0, g) := u_\star (0,\tau^*_{z_0} g).\]
It follows that
    \[0 = J(\tau^*_{z_0} g, u_\star) = J(g, z_0 + u_\star),\]
and hence $\Sigma_{z_0 + u_\star(z_0,g)}$ has QPMC with respect to $g$. Moreover, $u_\star$ is the unique element of 
    \[\{u \in \mathring{\RC}^{\ell,\gamma} (\S^{n-k};\R^k) : \|u\|_{\RC^{\ell,\gamma}} \leq \delta_0\}\]
which has this property. Finally, by \eqref{bound u by g} we have 
    \[\|u_\star(z_0, g)\|_{\RC^{\ell,\gamma}} = \|u_\star(0, \tau^*_{z_0} g)\|_{\RC^{\ell,\gamma}} \leq C \|\tau^*_{z_0} g - g_0\|_{\RC^{\ell-1,\gamma}} =C \|g-g_0\|_{\RC^{\ell-1,\gamma}}.\]
\end{proof}

Next we demonstrate that if $g$ is smooth then a graphical submanifold with QPMC is smooth.

\begin{proposition}\label{regularity}
If $g$ is smooth and $u \in {\RC}^{3,\gamma}(\S^{n-k};\R^k)$ is such that 
    \[(1-Q_{g,u})(H_{g,u}) = 0,\]
then $u \in \RC^{\infty}(\S^{n-k};\R^k)$. 
\end{proposition}
\begin{proof}
The QPMC equation 
    \[H = Q(H)\]
asserts that $H$ is a linear combination of finitely many eigensections of $-\Delta^\perp$. If $u$ is of class $\RC^{\ell,\gamma}$ for some $\ell \geq 3$ then each of these eigensections is of class $\RC^{\ell-1,\gamma}$ by Lemma~\ref{eigensection regularity}. Moreover, in appropriate local coordinates, $H$ is given by a uniformly elliptic quasilinear operator acting on $u$, whose coefficients are of class $\RC^{\ell-1,\gamma}$ if $u \in \RC^{\ell,\gamma}$. Therefore, the Schauder estimates (see e.g. \cite[Section~5.5]{Giaquinta}) imply that if $u \in \RC^{\ell,\gamma}$ then $u \in \RC^{\ell+1,\gamma}$. A standard bootstrapping argument thus shows that $u$ is smooth if it is assumed to be of class $\RC^{3,\gamma}$. 
\end{proof}

We may now prove that the submanifolds $\Sigma_{z+u_\star}$ form a smooth foliation of $M$ as $z$ ranges over $\R^k$. 

\begin{corollary}\label{foliation existence 2}
Suppose we are in the same setting as in Proposition~\ref{foliation existence 1}. If $g$ is smooth and $\varepsilon$ is sufficiently small then the map $\Phi_g : M \to M$ given by
    \[\Phi_g(z,\omega) := (z + u_\star(z, g)(\omega), \omega)\]
is a diffeomorphism. In particular, the submanifolds 
    \[\Sigma_{z + u_\star(z,g)} = \Phi_g(\{z\} \times \S^{n-k})\]
form a smooth foliation of $M$. 
\end{corollary}
\begin{proof}
The assignment
    \[g \mapsto u_\star(0, g)\]
is a smooth map from $\mathcal G$ to $\mathring{\RC}^{\ell,\gamma}(\S^{n-k};\R^k)$ and, since we are assuming $g$ is smooth, 
    \[z \mapsto \tau^*_z g\]
is a smooth map from $\mathbb{R}^k$ into $\mathcal G$. Combining these facts with Lemma~\ref{regularity}, we see that
    \[(z,\omega) \mapsto u_\star(z,g)(\omega)\]
is a smooth map from $M$ to $\mathbb{R}^k$, and hence $\Phi$ is smooth. By Proposition~\ref{foliation existence 1} we have
    \[\|\Phi_g - \Id\|_{\RC^1(M)} \leq C \varepsilon.\]
Therefore, if $\varepsilon$ is sufficiently small then $\Phi_g$ is a diffeomorphism. 
\end{proof}

We conclude this section with the proof of Theorem~\ref{main entire}.

\begin{proof}[Proof of Theorem~\ref{main entire}]
The existence component of the theorem follows immediately from Proposition~\ref{foliation existence 1} and Corollary~\ref{foliation existence 2}. To prove the uniqueness component, suppose $\Sigma$ is an embedded $(n-k)$-sphere in $M$ which has QPMC and is a graph, 
    \[\Sigma = \{(z+u(\omega), \omega) : \omega \in \S^{n-k}\},\]
such that
    \[\|u\|_{\RC^{\ell,\gamma}} \leq \delta.\]
Defining
    \[\bar u := \frac{1}{|\S^{n-k}|} \bigg(\int_{\S^{n-k}} u^1 \, \dvoll_{\S^{n-k}}, \dots, \int_{\S^{n-k}} u^k \, \dvoll_{\S^{n-k}}\bigg),\]
we have $u - \bar u \in \mathring{\RC}^{\ell,\gamma}(\S^{n-k};\R^k)$ and $\|u - \bar u\|_{\RC^{\ell,\gamma}} \leq 2\delta$. If $\varepsilon$ and $\delta$ are both sufficiently small then, by Proposition~\ref{foliation existence 1}, we have 
    \[u-\bar u = u_\star(z+\bar u, g),\]
which means $\Sigma$ is a leaf of the foliation provided by Proposition~\ref{foliation existence 1}. This completes the proof. 
\end{proof}


\section{Uniqueness and local foliations}\label{section uniqueness}

Let $g$ be a smooth metric on $M := \R^k \times \S^{n-k}$. We write $\class_g(\delta)$ for the class of all embedded $(n-k)$-spheres in $M$ which are $\delta$-vertical in the sense of Definition~\ref{delta vertical}. Note that when $g$ is close to $g_0$ the scale $\bar r$ appearing in that definition is close to $1$. 

Before proceeding with the main results of this section we establish some lemmas. First, elements of $\class_g(\delta)$ are graphical. 

\begin{lemma}\label{vertical implies graphical}
There are positive constants $\varepsilon$ and $\delta$, depending only on $n$, with the following property. Let $g$ be a metric on $M$ such that $\|g - g_0\|_{\RC^3} \leq \varepsilon$, and suppose $\Sigma \in \class_g(\delta)$. Then $\Sigma$ is the graph of a function $u:\S^{n-k} \to \R^k$. Moreover, there is a constant $C = C(n)$ such that 
    \begin{equation}\label{delta vertical C^3}
    \|du\|_{C^3(\S^{n-k})} \leq C\bigg(\|g-g_0\|_{C^{3}} +  \sum_{\ell = 0}^2  \sup_\Sigma |(\nabla^\perp)^\ell A_\Sigma|\bigg).
    \end{equation}
\end{lemma}
\begin{proof}
We first claim that $\Sigma$ is graphical. Let us write $q_{ab}$ for the matrix $g(\partial_{z^a}^\perp, \partial_{z^b}^\perp)$, where $1 \leq a, b \leq k$. It suffices to show that $\det(q_{ab}) > 0$ at every point of $\Sigma$. Let us define
    \[u^a := z^a|_{\Sigma}.\]
We insert
    \[\grad_\Sigma u^a = (\grad z^a)^\top = \grad z^a - (\grad z^a)^\perp\]
into the identity
    \[\Hess_\Sigma u (X,Y) = h(\nabla^\top_X(\grad_\Sigma u), Y) = g(\nabla_{X}(\grad_\Sigma u), Y),\]
in order to obtain
    \begin{equation}\label{Hessian of height}
        \Hess_\Sigma u^a(\cdot, \cdot) = \Hess z^a(\cdot, \cdot) + g(A_\Sigma(\cdot,\cdot), (\grad z^a)^\perp).
    \end{equation}
We then deduce the estimate
    \begin{equation}\label{delta vertical hessian}
        \sup_\Sigma |\Hess_\Sigma u^a| \leq C(\|g-g_0\|_{\RC^1} + \sup_\Sigma |A_\Sigma|).
    \end{equation}
There exists a maximal point for $u^a$ in $\Sigma$, at which  $\grad_\Sigma u^a = 0$. Using the previous inequality and integrating along geodesics from this maximal point, we obtain
    \begin{equation}\label{delta vertical gradient}
    \sup_\Sigma |\grad_\Sigma u^a| \leq C(\|g-g_0\|_{\RC^1} + \sup_\Sigma |A_\Sigma|),
    \end{equation}
where we have used the $\delta$-vertical condition to absorb $\diam_g(\Sigma)$ into $C$. We combine this inequality with the identity
    \[g((\grad z^a)^\perp, (\grad z^b)^\perp) = g(\grad z^a, \grad z^b) - g(\grad_\Sigma u^a, \grad_\Sigma u^b).\]
Up to errors which can be made arbitrarily small by choosing $\varepsilon$ and $\delta$ sufficiently small, the left-hand side is $q_{ab}$ and the right-hand side is $\delta_{ab}$. Consequently, $\det(q_{ab}) > 0$ if $\varepsilon$ and $\delta$ are sufficiently small. 

We may therefore express $\Sigma$ as a graph over $\S^{n-k}$. Conflating $u^a$ with its pullback to $\S^{n-k}$, we have 
    \[\Sigma = \{(u(\omega), \omega) : \omega \in \S^{n-k}\}.\]
The estimates \eqref{delta vertical hessian} and \eqref{delta vertical gradient}  imply 
    \[\|du\|_{\RC^1(\S^{n-k})} \leq C(\|g-g_0\|_{\RC^1} + \sup_\Sigma |A_\Sigma|).\]
By differentiating \eqref{Hessian of height} twice we get additional estimates for the third and fourth derivatives of $u$, such that \eqref{delta vertical C^3} holds.
\end{proof}

The next statement is a kind of converse to Lemma~\ref{vertical implies graphical}, asserting that sufficiently controlled graphs are in $\class_g(\delta)$. 

\begin{lemma}\label{graphical implies vertical}
Consider a smooth map $u : \S^{n-k} \to \R^k$ and let $\Sigma_u$ denote its graph. Given $\delta > 0$, if $\|g-g_0\|_{\RC^3}$ and $\|du\|_{\RC^3(\S^{n-k})}$ are sufficiently small (depending only on $n$ and $\delta)$ then $\Sigma_u \in \class_g(\delta)$. 
\end{lemma}
\begin{proof}
Consider a map $u : \S^{n-k} \to \R^k$ and let $\Sigma$ denote the graph of $u$. If $\|g-g_0\|_{\RC^0}$ and $\|du\|_{\RC^0(\S^{n-k})}$ are both small then the diameter of $\Sigma$ is close to $\pi$. The identity
    \begin{equation}\label{A of graph}
        g(A_\Sigma(\cdot,\cdot), (\grad z^a)^\perp) = - \Hess z^a(\cdot, \cdot) + \Hess_\Sigma u^a(\cdot, \cdot)
    \end{equation}
implies 
    \[\sup_\Sigma |A_\Sigma| \leq C(\|g-g_0\|_{\RC^1} + \|du\|_{\RC^1(\S^{n-k})}).\]
Differentiating \eqref{A of graph} twice leads to an estimate of the form 
    \begin{equation*}
    \sum_{\ell = 0}^2  \sup_\Sigma |(\nabla^\perp)^\ell A_\Sigma| \leq C(\|g-g_0\|_{C^{3}} + \|du\|_{C^3(\S^{n-k})}),
    \end{equation*}
so if $\|g-g_0\|_{\RC^3}$ and $\|du\|_{\RC^3(\S^{n-k})}$ are small then $\Sigma \in \class_g(\delta)$. 
\end{proof}

Recall that Proposition~\ref{foliation existence 1} and Corollary~\ref{foliation existence 2} established the existence of a QPMC foliation for metrics $g$ close to $g_0$, given by $\Sigma_{z + u_\star(z,g)}$ for $z \in \R^k$. We henceforth refer to this as the canonical foliation for $g$. 

If $g$ is sufficiently close to $g_0$ then the leaves of its canonical foliation are in $\class_g(\delta)$. 

\begin{lemma}\label{foliation = delta vertical}
Given $\delta >0$ there exists $\varepsilon = \varepsilon(n,\delta)$ such that if $\|g-g_0\|_{\RC^4} \leq \varepsilon$ then the leaves of the canonical foliation for $g$ are in $\class_g(\delta)$. 
\end{lemma}
\begin{proof}
For each leaf $\Sigma_{z + u_\star}$ we have 
    \[\|du_\star\|_{\RC^3(\S^{n-k})} \leq C\|g-g_0\|_{\RC^4} \leq C\varepsilon,\]
by Proposition~\ref{foliation existence 1}. Therefore, if $\varepsilon$ is small then Lemma~\ref{graphical implies vertical} guarantees that $\Sigma_{z + u_\star} \in \class_g(\delta)$. 
\end{proof}

The next lemma strengthens the uniqueness statement in Proposition~\ref{foliation existence 1}.

\begin{lemma}\label{extrinsic uniqueness}
There are positive constants $\varepsilon$ and $\delta$, depending only on $n$, with the following property. Let $g$ be a metric on $M$ such that $\|g - g_0\|_{\RC^3} \leq \varepsilon$ and suppose $\Sigma \in \class_g(\delta)$ has QPMC. Then $\Sigma$ belongs to the canonical foliation for $g$.
\end{lemma}
\begin{proof}
By Lemma~\ref{vertical implies graphical}, we may assume $\delta$ is sufficiently small so that $\Sigma$ is the graph of a smooth map $u : \S^{n-k} \to \R^k$. We define 
    \[z := \frac{1}{|\S^{n-k}|}\bigg(\int_{\S^{n-k}} u^1  \, \dvoll_{\S^{n-k}},\dots, \int_{\S^{n-k}} u^k \, \dvoll_{\S^{n-k}}\bigg),\]
so that the components of $u-z$ are mean-zero functions. The derivative bounds for $u$ in Lemma~\ref{vertical implies graphical} then imply
    \[\|u-z\|_{\RC^4(\S^{n-k})} \leq C(\delta + \varepsilon).\]
If $\delta$ and $\varepsilon$ are both small then by the uniqueness statement in Proposition~\ref{foliation existence 1} we conclude that
    \[u-z = u_\star(0,\tau_z^*g) = u_\star(z,g).\]
The claim follows.
\end{proof}

We are now prepared to prove a local version of Theorem~\ref{main entire}. 

\begin{proposition}\label{local foliation}
For every $\delta > 0$ there exists a constant $\varepsilon = \varepsilon(n,\delta)$ with the following property. Let $g$ be a smooth Riemannian metric on $B^k(0,L) \times \S^{n-k}$, where $L \geq 1000$. If $\|g-g_0\|_{\RC^4} \leq \varepsilon$ then there is an open subset
    \[B^k(0,L-20) \times \S^{n-k}\subset \mathcal B \subset B^k(0,L-10) \times \S^{n-k}\]
which is foliated by $(n-k)$-spheres of class $\class_g(\delta)$ with QPMC. Moreover, for sufficiently small $\delta$ (depending only on $n$) we have the following uniqueness statement: if $\Sigma \subset \CB$ is of class $ \class_g(\delta)$ and has QPMC, then $\Sigma$ is a leaf of the foliation. 
\end{proposition}
\begin{proof}
We may extend $g$ to a smooth metric on $M$ such that 
    \[\|g - g_0\|_{\RC^4(M)} \leq C \|g-g_0\|_{\RC^4(B^k(0,L) \times \S^{n-k})} \leq C\varepsilon\]
for some universal $C$. If $\varepsilon$ is small enough then  Corollary~\ref{foliation existence 2} provides a canonical foliation for the extension $g$. (The uniqueness statement in Proposition~\ref{foliation existence 1} ensures that the portion of this foliation inside $B^k(0,L-10) \times \S^{n-k}$ does not depend on how we extended $g$.) By Lemma~\ref{foliation = delta vertical}, if $\varepsilon$ is sufficiently small then the leaves of the foliation are in $\class_g(\delta)$. It then suffices to take $\CB$ to be a union of leaves such that 
    \[B^k(0,L-20) \times \S^{n-k}\subset \mathcal B \subset B^k(0,L-10) \times \S^{n-k}.\]
Now suppose $\Sigma \subset \CB$ is of class $ \class_g(\delta)$ and has QPMC. By Lemma~\ref{extrinsic uniqueness}, if $\delta$ is sufficiently small then $\Sigma$ is a leaf of the canonical foliation for $g$, and hence is a leaf of the foliation for $\CB$.
\end{proof}

We conclude this section with the proof of Theorem~\ref{main local}.

\begin{proof}
Fix a constant $L \geq 1000$ and let $\CC \subset M$ be an $(\varepsilon, L, n-k)$-cylindrical region of a complete Riemannian manifold $(M,g)$. We recall what this means: for each $p \in \CC$ there is a scale $r(p) > 0$ and an embedding 
    $F_p : B^k(0,L) \times \S^{n-k} \to M$
such that $F_p^*(r(p)^{-2}g)$ is $\varepsilon$-close to $g_0$ in $\RC^{[1/\varepsilon]}$. 

By Proposition~\ref{local foliation}, if $\varepsilon$ is sufficiently small then for each $p \in \CC$ there is an open set 
    \[F_p(B^k(0,L-20) \times \S^{n-k}) \subset \mathcal B^+_p \subset F_p(B^k(0,L-10)\times\S^{n-k})\]
which is foliated by embedded $(n-k)$-spheres with QPMC. We denote the collection of leaves of this foliation by $\mathcal F_p^+$. For any $\delta>0$, we can assume $\varepsilon$ is sufficiently small that the leaves of $\mathcal F_p^+$ are $\delta$-vertical (because of Lemma~\ref{foliation = delta vertical}). For each $p \in \CC$ we now define $\mathcal \CB_p$ to be a connected open subset of $\CB_p^+$ which is a union of leaves in $\mathcal F_p^+$ and satisfies 
    \[F_p(B^k(0,80) \times \S^{n-k}) \subset \mathcal B_p \subset F_p(B^k(0,90)\times\S^{n-k}).\]
Let $\mathcal F_p$ denote the collection of leaves in $\mathcal F_p^+$ which are contained in $\mathcal B_p$. 

We define $\mathcal B  := \cup_{p \in \CC}\mathcal B_p$ and observe that $\CB$ is an open subset of $M$ which contains $\CC$. We claim that the foliations $\mathcal F_p$ can be adjoined into a foliation of $\mathcal B$. So we consider points $p$ and $q$ such that $\CB_p \cap \CB_q$ is nonempty. For any $x \in \CB_p \cap \CB_q$, let $\Sigma_q(x)$ denote the leaf of $\mathcal F_q$ passing through $x$. The pullback of $\Sigma_q(x)$ by $F_p$ is of class $\class(\delta)$ in the metric $F_p^*(r(p)^{-2}g)$, so if $\delta$ is small enough we may invoke the uniqueness statement in Proposition~\ref{local foliation} to conclude that $\Sigma_q(x) \in \mathcal F_p^+$. Given that $x \in \mathcal{B}_p$, we therefore have $\Sigma_q(x) \subset \CB_p$ or equivalently $\Sigma_q(x) \in \mathcal F_p$. Since $x$ was arbitrary it follows that $\CB_p \cap \CB_q$ is a union of leaves of $\mathcal F_q$, each of which is also in $\mathcal F_p$. This is true for any $p$ and $q$ with $\CB_p \cap \CB_q$ nonempty, so $\cup_p \mathcal F_p$ is indeed a foliation for $\mathcal B$.

To conclude, suppose $\Sigma$ is an embedded $(n-k)$-sphere in $M$ of class $\mathscr{S}(\delta)$, which has QPMC and intersects $\mathcal C$ at some point $p$. Then $\Sigma$ belongs to $\mathcal F_p$ by the uniqueness statement in Proposition~\ref{local foliation}. This completes the proof.
\end{proof}

\section{Metrics with no PMC foliation}\label{section example}

A submanifold $\Sigma$ of an ambient space $(M,g)$ has parallel mean curvature (PMC) if $\nabla^\perp H = 0$ holds everywhere on $\Sigma$. 

Theorem~\ref{main entire} asserts that if a metric on $\R^k \times \S^{n-k}$ is close enough to the standard one, then the slices $\{z\}\times\S^{n-k}$ can be perturbed to a canonical foliation by QPMC-spheres. It is natural to wonder whether there is in fact a foliation by spheres with PMC (or, rather, since the QPMC foliation is unique, and PMC implies QPMC, whether the leaves of our canonical foliation actually have PMC). We demonstrate that this is not the case in general---there are metrics on $\R^2 \times \S^1$ which are arbitrarily close to the standard one, but cannot be foliated by approximately vertical spheres with PMC. 

Let us sketch the construction of such a metric. We write $M = \R^2 \times \S^1$ and as usual denote the standard product metric on $M$ by $g_0$. For each metric $g$ on $M$ which is $\RC^4$-close to $g_0$, we have a canonical QPMC foliation by Theorem~\ref{main entire}. For each $(z,\omega) \in M$ we denote by $\Sigma_g(z,\omega)$ the leaf of this foliation which contains $(z,\omega)$. 

\begin{lemma}\label{example Berger}
There exists a metric $g$ on $M$ which is $\RC^4$-close to $g_0$ such that $g$ has positive sectional curvatures in $\{|z| \leq 110\}$ and $\Sigma_g(z,\omega)$ is a stable geodesic for every $|z| \leq 100$.
\end{lemma}
\begin{proof}
To begin with consider $\S^{3}$ equipped with the unit round metric and identified with $SU(2)$. Let $E_\alpha$ be a left-invariant orthonormal frame, and $\omega^\alpha$ its dual frame. For each $\kappa > 0$ we have a Berger metric
    \[\omega^1\otimes\omega^1 + \omega^2\otimes\omega^2 + \kappa^2 \omega^3 \otimes \omega^3.\]
We will find it convenient to scale this metric by $\kappa^{-2}$ and define
    \[g_\kappa := \kappa^{-2} \omega^1\otimes\omega^1 + \kappa^{-2} \omega^2\otimes\omega^2 + \omega^3 \otimes \omega^3.\]
The sectional curvatures of $g_\kappa$ are 
    \[K(E_1, E_2) = \kappa^2(4-3\kappa^2), \qquad K(E_1, E_3) = \kappa^4, \qquad K(E_2,E_3) = \kappa^4.\]
The integral curves of $E_3$ are closed geodesics of length $2\pi$ with respect to $g_\kappa$, and we may assume $\kappa$ is sufficiently small so that each of these geodesics is stable (see \cite[Proposition~6]{Torralbo--Urbano}).

Fix any point $p \in \S^3$. We identify the unit-speed geodesic tangent to $E_3$ and passing through $p$ with $\S^1$. Consider then the map $M \to \S^3$ given by 
    \[(z, \omega) \mapsto \exp_\omega(\kappa z^1 E_1 + \kappa z^2 E_2)\]
where the exponential map is with respect to $g_\kappa$. The pullback of $g_\kappa$ by this map converges in $\RC^\infty_{\loc}$ to $g_0$ as $\kappa \to 0$, so let us suppose $\kappa$ is small and define a new metric $g$ defined everywhere in $M$ which agrees with the pullback of $g_\kappa$ in $\{|z| \leq 110\}$ and is globally $\RC^4$-close to $g_0$. Then $g$ has positive sectional curvatures in $\{|z| \leq 110\}$. Moreover, by uniqueness of the QPMC foliation for $g$, if $\kappa$ is small enough then $\Sigma_g(z,\omega)$ is a stable geodesic for every $|z| \leq 100$, obtained as the pullback of an integral curve of $E_3$.
\end{proof}

Combined with the second variation formula for length, Lemma~\ref{example Berger} has the following consequence. 

\begin{lemma}\label{example stable}
There exists a metric on $M$ which is $\RC^4$-close to $g_0$ such that, for $|z| \leq 100$, the normal bundle of $\Sigma_g(z,\omega)$ admits no nonzero global parallel sections. In particular, for each of these submanifolds $\ker(\Delta^\perp) = \{0\}$. 
\end{lemma}
\begin{proof}
It suffices to take $g$ to be a metric with the properties referred to in Lemma~\ref{example Berger}. Then for every $|z| \leq 100$, the leaf $\Sigma = \Sigma_g(z,\omega)$ is a stable geodesic. In other words,
    \[\int_{\Sigma} |\nabla^\perp V|^2 - g(R(V,X)X,V)\dvol_{\Sigma} \geq 0\]
for every section $V$ of $N\Sigma$, where $X$ is a global unit tangent to $\Sigma$. We have 
    \[g(R(V,X)X,V) = K(X,V)|V|^2,\]
so since the sectional curvatures of $g$ are positive for $|z| \leq 110$, we conclude that $N\Sigma$ admits no nonzero global parallel sections.
\end{proof}

Now let $g$ be a metric with the property referred to in Lemma~\ref{example stable}. Let $g_j$ be a sequence of bumpy metrics which converges to $g$ in the $\RC^4$-norm. (A metric is called bumpy if all of its minimal submanifolds are nondegenerate critical points of the volume functional; this is a generic property by \cite{White_bumpy}.) For $j$ large enough, we have an associated QPMC foliation given by $\Sigma_{g_j}(z,\omega)$. Since $g_j$ is bumpy, for any open set $\Omega \subset \R^2$, there must be some $z \in \Omega$ such that $\Sigma_{g_j}(z,\omega)$ is not a geodesic---otherwise, we have a foliation by geodesics, each of which carries a nonzero Jacobi field (obtained from any variation through leaves) and is thus degenerate. So let $(z_j, \omega_j)$ be a sequence of points in $\{|z| \leq 90\}$ such that $\Sigma_j := \Sigma_{g_j}(z_j,\omega_j)$ is not a geodesic. After passing to a subsequence, $\Sigma_j$ converges to $\Sigma_g(z,\omega)$ for some $(z,\omega) \in \{|z| \leq 100\}$. By Lemma~\ref{example stable}, for sufficiently large $j$ we know that
    \[\ker(\Delta^\perp_{\Sigma_j}) = \ker(\Delta^\perp_{\Sigma_g(z,\omega)}) = \{0\},\]
so $N\Sigma_j$ admits no nonzero global parallel sections. But we know that the curvature vector of $\Sigma_j$ does not vanish identically, so it cannot be parallel for large $j$. To summarise, when $j$ is large the QPMC foliation for $g_j$ contains at least one leaf which has QPMC but not PMC. (In fact, our argument indicates that `most' leaves of the foliation for $g_j$ have QPMC but not PMC.)

\section{Bubblesheets in mean curvature flow}\label{section MCF}

In this section we indicate how bubblesheet regions arise along the mean curvature flow, and how our QPMC foliation can be used to canonically parameterize these regions. 

Consider a closed hypersurface $M_0 \subset \mathbb{R}^{n+1}$, and let $\{M_t\}_{t \in [0,T)}$ be the maximal smooth evolution of $M_0$ by the mean curvature flow. We write $\lambda_1 \leq \dots \leq \lambda_n$ for the principal curvatures of $M_t$, so that its (scalar) mean curvature is given by 
    \[H = \lambda_1 + \dots + \lambda_n.\]
We assume $M_0$ is $(m+1)$-convex for some $1 \leq m \leq n-1$, meaning that 
    \[\lambda_1 + \dots + \lambda_{m+1} > 0\]
everywhere on $M_0$. (When $m = n - 1$ this condition simply says that $M_0$ has positive mean curvature.) It follows that $M_t$ is uniformly $(m+1)$-convex for all $t \in [0,T)$, in the sense that 
    \[\min_{M_t} \, \frac{\lambda_1 + \dots + \lambda_{m+1}}{H} \geq \min_{M_0} \, \frac{\lambda_1 + \dots + \lambda_{m+1}}{H} >0.\]
The maximal time $T < \infty$ is characterized by 
    \[\lim_{t\to T}\bigg(\max_{M_t} H\bigg) = \infty.\]
For these results and further background see e.g. \cite[Chapter~9]{Andrews_etc}.

We recall a well-known variant of the Neck Detection Lemma for 2-convex solutions. In short, wherever the curvature of $M_t$ is large, a certain pointwise inequality for its principal curvatures at a single point is enough to ensure the existence of a large almost cylindrical region. For the proof, see \cite{Haslhofer--Kleiner}. 

\begin{proposition}\label{Bsheet detection}
Given $L > 0$ and $\varepsilon > 0$ there exist positive constants $\delta = \delta(n, M_0, L, \varepsilon)$ and $K = K(n, M_0, L, \varepsilon)$ with the following property. If $x \in M_t$ is such that $H(x,t) \geq K$ and 
   \[\frac{\lambda_1 + \dots + \lambda_{m}}{H}(x,t) \leq \delta\]
then the shifted and scaled hypersurface
    \[\frac{H(x,t)}{n-m}(M_t - x) \cap B(0,L)\]
is $\varepsilon$-close in $\RC^{[1/\varepsilon]}$ to (a rotation of) $\mathbb{R}^{m} \times \S^{n-m}$.
\end{proposition}

Let us consider a fixed $\varepsilon > 0$ and $L \geq 10^6$. We define 
    \[\CC_t = \bigg\{x \in M_t :  H(x,t) \geq K  \text{ and } \frac{\lambda_1 + \dots + \lambda_m}{H}(x,t) \leq \delta\bigg\},\]
where $K$ and $\delta$ are as in Theorem~\ref{Bsheet detection}. For any $\varepsilon' >0$, if $\varepsilon$ is sufficiently small relative to $\varepsilon'$, then $\CC_t$ is $(\varepsilon', L/10, n-m)$-cylindrical in the sense of Definition~\ref{cylindrical region}.

By Theorem~\ref{main local}, $\CC_t$ is contained in an open set $\CB_t$ which admits a canonical foliation by approximately vertical $(n-m)$-spheres with QPMC. Using the fact that $\CB_t$ is also a hypersurface in $\mathbb{R}^{n+1}$, we can associate with it a core in $\R^{n+1}$, defined as follows. Writing $\Sigma(x)$ for the QPMC leaf passing through $x \in \CB_t$, we set 
    \[\Psi_t(x) := \frac{1}{|\Sigma(x)|} \int_{\Sigma(x)} y \dvoll_{\Sigma(x)}.\]
That is, $\Psi_t$ sends $\Sigma(x)$ to its center of mass in $\mathbb{R}^{n+1}$. If $\varepsilon$ is sufficiently small then $\Gamma_t := \Psi_t(\CB_t)$ is a smooth $m$-dimensional submanifold of $\mathbb{R}^{n+1}$, the map $\Psi_t$ is a submersion from $\CB_t$ to $\Gamma_t$, and the curvature of $\Gamma_t$ at $\Psi_t(x)$ is much smaller than $H(x,t)$. 

The region $\CB_t$ is completely described by its QPMC foliation, center-of-mass map $\Psi_t$ and core $\Gamma_t$. Together, these pieces of data play the role of Hamilton's normal form for necks (see Section~3 of \cite{Hamilton_PIC}). 

We intend for our normal form description of bubblesheet regions to find many applications. We give only one example here---a mechanism for determining the topology of $M_0$ in case the curvature of $M_t$ becomes large everywhere at the same time. Let us write $\Omega_t$ for the region enclosed by $M_t$.

\begin{proposition}\label{bubblesheet singularities}
Suppose $\CC_t = M_t$. Then $\Gamma_t$ is closed and $\Omega_t$ is smoothly isotopic to a tubular neighbourhood of $\Gamma_t$. In particular, $\Omega_0$ is isotopic to a tubular neighbourhood of a closed $m$-dimensional submanifold of $\mathbb{R}^{n+1}$. 
\end{proposition}
\begin{proof}
We have $\CB_t = M_t$. Since $M_t$ is closed and $\Psi_t$ is a submersion, $\Gamma_t$ is closed. Provided $\varepsilon$ is sufficiently small, each leaf of the QPMC foliation $\Sigma(x)$ is almost round and almost lies in the affine $(n-m+1)$-dimensional subspace normal to $\Gamma_t$ at $\Psi_t(x)$. Performing a small deformation of each leaf we see that $\Omega_t$ is isotopic to a tubular neighbourhood of $\Gamma_t$. This proves the first claim. The second claim follows since $\Omega_0$ is isotopic to $\Omega_t$. 
\end{proof}

It is natural to ask which tubular neighbourhoods in $\R^{n+1}$ actually arise in the conclusion of Proposition~\ref{bubblesheet singularities}. There is reason to believe that all of them do, up to isotopy. Indeed, a sufficiently thin tubular neighbourhood of \emph{any} closed $m$-submanifold of $\R^{n+1}$ is $(m+1)$-convex. Flowing the boundary of such a thin tubular neighbourhood might, in many cases, lead to a singularity where the normal fibres pinch off everywhere simultaneously. 

\appendix

\section{First variation of \texorpdfstring{$\Delta^\perp$}{the normal Laplacian}}\label{appendix commutators}

Let $(M,g)$ be a Riemannian $n$-manifold, and let $\Sigma$ be a compact codimension-$k$ submanifold of $M$. As in Section~\ref{section variations}, we consider a smooth variation of $\Sigma$, given by a smooth family of submanifolds $\Sigma_s$ such that $\Sigma_0 = \Sigma$. We parameterize the variation by a family of embeddings
    \[
    F : \Sigma \times (-\varepsilon, \varepsilon) \to M
    \]
such that $F(\cdot, 0)$ is the inclusion $\Sigma \hookrightarrow M$. We also assume that $\tfrac{\partial F}{\partial s}(\cdot, 0)$ is normal to $\Sigma_0$ and define
    \[V := \frac{\partial F}{\partial s}(\cdot, 0).\]

Let $W = W(x,s)$ be a smoothly varying normal vector field on $\Sigma_s$. In this appendix we give an explicit expression for the commutator    
    \[\Lambda(V, W(\cdot, 0)) := (\nabla_s^\perp \Delta^\perp W - \Delta^\perp \nabla_s^\perp W)_{s = 0},\]
which appears in Proposition~\ref{var P}. 

We compute at a point $p$ in $\Sigma$, which we may assume lies at the origin in a system of geodesic normal coordinates $x^i$ on $\Sigma$. Let $X_i = \frac{\partial F}{\partial x^i}$.

\begin{lemma}\label{commutator_gradient}
At the point $p$, when $s = 0$ we have 
    \begin{align*}
        \nabla_{s}^\perp(\nabla_i^\perp W)&-\nabla_i^\perp(\nabla_s^\perp W)=g(W, A^j_i)\nabla_j^\perp V - g(W, \nabla^\perp_j V) A^j_i +(R(V, X_i)W)^\perp.
    \end{align*}
\end{lemma}
\begin{proof}
Let us fix a smooth orthonormal frame $N_a = N_a(s)$ for the normal space to $\Sigma_s$ at $F(p,s)$. We may assume that each $N_a$ satisfies 
    \[\nabla_s^\perp N_a = 0.\]
It then follows that
    \[\nabla_s N_a = (\nabla_s N_a)^\top = -g(N_a, \nabla_j^\perp V) h^{jk} X_k 
    \]
when $s = 0$. We may write
    \[\nabla_i^\perp W = \sum_{a} g(\nabla_{X_i} W, N_a) N_a.\]
Consequently, when $s = 0$ we have
 \begin{equation}   
    \begin{aligned}\label{commutator_gradient_1}
        \nabla_s^\perp(\nabla_i^\perp W) &= \bigg(\frac{\partial}{\partial s} \sum_{a} g(\nabla_{X_i} W, N_a) N_a\bigg)^\perp \\
        &= \bigg(\frac{\partial}{\partial s} \nabla_{X_i} W\bigg)^\perp + \sum_a g(\nabla_{X_i} W, \nabla_s N_a) N_a \\
        &= \bigg(\frac{\partial}{\partial s} \nabla_{X_i} W\bigg)^\perp + g(W, A_i^j)\nabla_j^\perp V.
    \end{aligned}
\end{equation}
We now compute the first term on the right-hand side of \eqref{commutator_gradient_1}. Let us fix a frame $E_\alpha$ for $TM$ in a neighbourhood of $p$ such that $\nabla_{E_\alpha} E_\beta = 0$ at $p$. We have 
    \begin{align*}
        \nabla_{X_i} W = \frac{\partial W}{\partial x^i} +  W^\alpha \nabla_{X_i} E_\alpha,
    \end{align*}
and hence when $s = 0$
    \begin{align*}
        \frac{\partial}{\partial s} \nabla_{X_i} W &= \frac{\partial}{\partial x^i} \frac{\partial W}{\partial s} + W^\alpha  \nabla_{\frac{\partial F}{\partial s}}  \nabla_{X_i} E_\alpha.
    \end{align*}
Inserting 
    \begin{align*}
        \frac{\partial W}{\partial s} &= \nabla_s W -  W^\alpha  \nabla_{\frac{\partial F}{\partial s}} E_\alpha\\
        &= \nabla_s^\perp W - g(W, \nabla_j^\perp V) h^{jk} X_k -  W^\alpha \nabla_{\frac{\partial F}{\partial s}} E_\alpha,
    \end{align*}
and projecting onto the normal space, we obtain 
    \begin{align*}
        \bigg(\frac{\partial}{\partial s} \nabla_{X_i} W\bigg)^\perp &= \nabla_i^\perp (\nabla_s^\perp W) - g(W, \nabla^\perp_j V) A^j_i +  W^\alpha( \nabla_{\frac{\partial F}{\partial s}} \nabla_{X_i} E_\alpha - \nabla_{X_i} \nabla_{\frac{\partial F}{\partial s}} E_\alpha)^\perp\\
        &= \nabla_i^\perp (\nabla_s^\perp W) - g(W, \nabla^\perp_j V) A^j_i +(R(V, X_i)W)^\perp.
    \end{align*}
Combining this with \eqref{commutator_gradient_1} gives the claim. 
\end{proof}

We now use Lemma~\ref{commutator_gradient} to compute $\Lambda(V,W(\cdot,0))$.

\begin{proposition}\label{commutator_Laplacian prop} 
At the point $p$, when $s = 0$ we have 
    \begin{align}\label{commutator_Laplacian eq}
        &\Lambda(V,W(\cdot,0)) \\
        &= 2g(V,A^{ij})\nabla_i^\perp \nabla_j^\perp W\\
        &+h^{ij}\nabla_i^\perp\Big(g(W, A_j^k)\nabla_k^\perp V - g(W, \nabla_k^\perp V) A_j^k + (R(V, X_j)W)^\perp\Big) \notag\\
        &+h^{ij}\Big(g(\nabla_j^\perp W, A_i^k) \nabla_k^\perp V - g(\nabla_j^\perp W, \nabla_k^\perp V)A_i^k + (R(V, X_i)\nabla_j^\perp W)^\perp\Big) \notag\\
        &+h^{ij}\Big(2g(\nabla_i^\perp A_j^k,V) + 2g(A_i^k, \nabla^\perp_j V)\Big)\nabla_k^\perp W \\
        &-h^{ij}\Big(g(\nabla_i^\perp H, V) + g(H, \nabla_i^\perp V)\Big)\nabla^\perp_j W.\notag
    \end{align}
\end{proposition}

\begin{proof}
First note that when $s = 0$ we have 
    \[\nabla_s^\perp(\Delta^\perp W) = 2g(V,A^{ij})\nabla_i^\perp \nabla_j^\perp W + h^{ij} \nabla_s^\perp(\nabla_i^\perp(\nabla_j^\perp W)) - h^{ij} \frac{\partial}{\partial s} (\nabla_i^\top X_j^k) \nabla^\perp_k W.\]
We apply Lemma~\ref{commutator_gradient} twice to rewrite the term 
    \[h^{ij} \nabla_s^\perp(\nabla_i^\perp(\nabla_j^\perp W))\]
The result is $\Delta^\perp \nabla_s^\perp W$ plus the second and third lines on the right-hand side of \eqref{commutator_Laplacian eq}. Next we rewrite  
    \[- h^{ij} \frac{\partial}{\partial s} (\nabla_i^\top X_j^k) \nabla^\perp_k W\]
using
    \[\nabla_i^\top X_j^k = \frac{1}{2}h^{kl}(X_i(h_{jl}) + X_j(h_{il}) - X_l(h_{ij}))\]
and 
    \[\frac{\partial h_{ij}}{\partial s} = -2g(A_{ij}, \tfrac{\partial F}{\partial s}^\perp).\]
This leads to the final two lines on the right-hand side of \eqref{commutator_Laplacian eq}. 
\end{proof}

\bibliographystyle{alpha}
\bibliography{references}

\end{document}